\documentclass[a4paper,10pt]{article}
\usepackage[english]{babel}
\usepackage{inputenc}
\usepackage[sumlimits]{amsmath}
\usepackage{amsthm}
\usepackage{amstext}
\usepackage{amssymb}
\usepackage{fancyhdr}
\usepackage{mathrsfs}
\usepackage{graphicx}
\usepackage[all]{xy}
\usepackage{enumerate}

\newtheorem{proposition}{Proposition}
\newtheorem{corollary}{Corollary}

\newtheorem{definition}{Definition}

\DeclareMathOperator{\N}{\mathbb{N}}
\DeclareMathOperator{\coinv}{coinv}
\DeclareMathOperator{\rec}{rec}
\DeclareMathOperator{\Z}{\mathbb{Z}}
\DeclareMathOperator{\R}{\mathbb{R}}

\DeclareMathOperator{\Rel}{R}
\DeclareMathOperator{\PPP}{\mathbb{P}}

\title{Combinatorial interpretations of particular evaluations of complete and elementary symmetric functions\thanks{This paper is part of the author's Ph.D.\ thesis written under the direction of Prof. F. Brenti at the Univ. "la Sapienza" of Rome, Italy.}}
\author{Pietro Mongelli}
\date{}
\begin{document}
\setcounter{page}{1}
\pagenumbering{arabic}
\maketitle
\begin{center}
 e-mail address: mongelli@mat.uniroma1.it
\end{center}


\\
{\bf Abstract} {\scriptsize 
The Jacobi-Stirling numbers and the Legendre-Stirling numbers of the first and second kind were first introduced in \cite{Everitt2002}, \cite{Everitt2007}. 
In this paper we note that Jacobi-Stirling numbers and Legendre-Stirling numbers are specializations of elementary and complete symmetric functions. We then study combinatorial interpretations of this specialization and obtain new combinatorial  interpretations of the Jacobi-Stirling and Legendre-Stirling numbers.}
\section{Introduction}
The aim of this paper is to give combinatorial interpretations of a family of numbers which includes the Legendre-Stirling numbers and Jacobi-Stirling numbers of both kinds.
The Jacobi-Stirling numbers were introduced  in \cite{Everitt2007} as the coefficients of the integral composite powers of the Jacobi differential operator
\begin{equation}
\label{E:JacobiOperator}
\mathfrak{l}_{\alpha,\beta}[y](t)=\frac{1}{(1-t)^{\alpha}(1+t)^\beta}\big(-(1-t)^{\alpha+1}(1+t)^{\beta+1}y'(t)\big)',
\end{equation}
with fixed real parameters $\alpha,\beta>-1$.
When the parameters are both equal to $0$, we find the definition of the Legendre-Stirling numbers, first introduced in \cite{Everitt2002} and later studied in \cite{Andrews2010}. In \cite{Everitt2002}, \cite{Everitt2007} and \cite{Mongelli2011} the authors show that both numbers share many properties with the classical Stirling numbers of both kinds such as similar recurrence relations, generating functions and total positivity properties. Recently several combinatorial interpretations of the  Legendre-Stirling numbers (\cite{Andrews2009},\cite{Egge2010}) and the Jacobi-Stirling numbers (\cite{Gelineau2010}) have been given, most of which are based on permutations and partitions, exactly the same combinatorial objects used for the classical interpretations of the Stirling numbers.

In this paper we note that the Jacobi-Stirling numbers and the Legendre-Stirling numbers of both kinds  are specializations of elementary and complete homogeneous symmetric functions. We study these specializations and then apply the results to Jacobi-Stirling and Legendre-Stirling numbers. 
More precisely, we give  general combinatorial interpretations of these specializations with a unified approach. These combinatorial interpretations include new combinatorial interpretations for the Jacobi-Stirling and Legendre-Stirling numbers as well as the results in \cite[Theorem 2]{Andrews2009}, \cite[Theorem 2.5]{Egge2010} and \cite[Theorem 7]{Gelineau2010}.

The organization of the paper is as follows.
In Section \ref{s:not} we recall the symmetric functions and some of their  properties and  show how we can obtain Jacobi-Stirling numbers and Legendre-Stirling numbers by a suitable evaluation of them. In Section \ref{S:Int1} we define a $q$-analogue of these numbers and we give a combinatorial interpretations of them that generalizes interpretations given in \cite{Andrews2009}, \cite{Egge2010}, \cite{Gelineau2010} when $q=1$ and that we can use in particular for the Jacobi-Stirling numbers. 
The evaluation of the symmetric functions that we use are parameterized by real nonnegative numbers. As done in \cite{Gelineau2010} for the Jacobi-Stirling numbers, in section \ref{S:Int2} we study our evaluations of the symmetric functions as polynomials in these parameters. 
In Section \ref{s:Dev} we turn our attention to other functions, already introduced in \cite{Brenti1995}, that  generalized the complete and elementary symmetric functions. We give a combinatorial interpretation of their evaluations and if we apply it to the case of Jacobi-Stirling numbers we get a new result. Moreover in this section we study the evaluations of the well known monomial symmetric functions.
Finally, in Section \ref{S:Rem} we recall other properties of the elementary and symmetric functions such that if applied to the Jacobi-Stirling numbers and Legendre-Stirling numbers gives us immediately some properties studied in \cite{Andrews2010}, \cite{Everitt2002} and \cite{Everitt2007}.

\section{Definitions, notation and preliminaries}
\label{s:not}
We let $\PPP:=\{1,2,3,...\}$, $\N:=\PPP\cup\{0\}$, $\Z=\N\cup\{-1,-2,-3,\dots\}$. The cardinality of a set $A$ will be denoted by $\vert A\vert$.

For the following definitions we use the notations of \cite[Chapter I.2]{Macdonald1979}. Consider the ring $\Z[x_1,\dots,x_n]$ of polynomials in $n$ independent variables $x_1,\dots, x_n$ with integer coefficients. For each $r\ge 0$ the $r$-th \emph{elementary symmetric function} $e_r$ is the sum of all products of $r$ distinct variables $x_i$, so that $e_0=1$ and for $r\ge1$
\begin{equation}
\label{E:ClassicalElemSymm}
e_r(x_1,\dots,x_n)=\sum_{i_1<i_2<\cdots<i_r}x_{i_1}x_{i_2}\cdots x_{i_r}.
\end{equation}
The $r$-th \emph{complete symmetric function} $h_r$ is the sum of all monomials of total degree $r$ in the variables $x_1,\dots,x_n$ so that $h_0=1$ and for $r\ge 1$
\begin{equation}
\label{E:ClassicalComplSymm}
h_r(x_1,\dots,x_n)=\sum_{i_1\le i_2\le \cdots\le i_r} x_{i_1}x_{i_2}\cdots x_{i_r}.
\end{equation}

It is a simple exercise to check the following recursion formulas for $n,j\ge 1$, $n> j$:
\begin{align}
h_{n-j}(x_1,\dots,x_j)=&h_{n-j}(x_1,\dots,x_{j-1})+x_j h_{n-j-1}(x_1,\dots,x_{j}); \label{E:Rich}\\
e_{n-j}(x_1,\dots,x_{n-1})=&e_{n-j}(x_1,\dots,x_{n-2})+x_{n-1}e_{n-j-1}(x_1,\dots,x_{n-2})\label{E:Rice}
\end{align}
with initial conditions
\begin{align}
h_0(x_1,\dots,x_j)&=1;&& h_{n-j}(x_1,\dots,x_j)=0 &\text{(if $n<j$)},\label{E:Inih}\\
e_0(x_1,\dots,x_{n-1})&=1; && e_{n-j}(x_{1},\dots,x_{n-1})=0 &\text{(if $n<j$)}.\label{E:Inie}
\end{align}

We now recall the definitions of the Jacobi-Stirling numbers and the Legendre-Stirling numbers and some of their basic properties.

The Jacobi-Stirling numbers of the second kind $JS_n^{(j)}(z)$ are defined for all $n,j\in\N$ by Everitt et al.\ via the following expansion of the $n$-th composite power of $\mathfrak{l}_{\alpha,\beta}$ (see \cite[Theorem 4.2]{Everitt2007}):
$$
(1-t)^{\alpha}(1+t)^{\beta}\mathfrak{l}_{\alpha,\beta}^{n}[y](t)=\sum_{j=0}^n(-1)^j\big(JS_n^{(j)}(\alpha+\beta+1)(1-t)^{\alpha+j}(1+t)^{\beta+j}y^{(j)}(t)\big)^{(k)},
$$
where $\mathfrak{l}_{\alpha,\beta}$ is the Jacobi differential operator (\ref{E:JacobiOperator}) and $\alpha,\beta>-1$ are fixed real parameters. Since these numbers depend only on the sum $\alpha+\beta$, we set $z=\alpha+\beta+1>-1$. 

In \cite[Section 4]{Everitt2007} or in \cite[Section 1]{Gelineau2010} the following recursive formula are given
\begin{align}
\label{E:RecJacStir2}
JS_n^{(j)}(z)&=JS_{n-1}^{(j-1)}(z)+j(j+z) JS_{n-1}^{(j)}(z)
\end{align}
with initial conditions
\begin{align}
\label{E:RecJacStir2ini}
JS_n^{(0)}(z)&=JS_{0}^{(j)}(z)=0;\quad JS_0^{(0)}(z)=1.
\end{align}

From relations (\ref{E:RecJacStir2}), (\ref{E:RecJacStir2ini})  can be easily deduced  \cite[Theorem 4.1]{Everitt2007}.

Moreover, in \cite{Everitt2007} the following equation is given:
\begin{equation}
\label{E:GfJS}
x^n=\sum_{j=0}^n JS_n^{(j)}\prod_{i=0}^{j-1}(x-i(z+i))
\end{equation}
and in the same paper the authors define the (unsigned) Jacobi-Stirling numbers of the first kind for all $n,j\in\N$ as follows
\begin{equation}
\label{E:GfJS1}
\prod_{i=0}^{n-1}(x-i(z+i)) =\sum_{j=0}^n (-1)^j Jc_n^{(j)}x^j.
\end{equation}

The Jacobi-Stirling numbers of the first kind satisfy the following recursion
\begin{align}
\label{E:RecJacStir1}
Jc_n^{(j)}(z)&=Jc_{n-1}^{(j-1)}(z)+(n-1)(n-1+z) Jc_{n-1}^{(j)}(z)
\end{align}
with initial conditions
\begin{align}
\label{E:RecJacStir1ini}
Jc_n^{(0)}(z)&=Jc_{0}^{(j)}(z)=0;\quad Jc_0^{(0)}(z)=1.
\end{align}

It is simple to check that recursion (\ref{E:Rich}) is the same as (\ref{E:RecJacStir2}) if $x_j=j(j+z)$ for all $j\ge 1$ and that recursion (\ref{E:Rice}) is the same as (\ref{E:RecJacStir1}) if $x_{n-1}=(n-1)(n-1+z)$ for all $n\ge 2$. 

The Legendre-Stirling numbers were first introduced  in \cite{Everitt2002} as the coefficients of the integral powers of the second-order Legendre differential operator $\mathfrak{l}[\cdot]$  (\ref{E:JacobiOperator}) when $\alpha=\beta=0$. Therefore these numbers coincide with the Jacobi-Stirling numbers when $z=1$.
As noted in \cite{Everitt2002} and \cite{Everitt2007}, both numbers have properties similar to the classical Stirling numbers of both kinds. This is  because the Stirling numbers also satisfy the same recursion formulas (\ref{E:Rich}) and (\ref{E:Rice}) with $x_j=j$ for all $j\ge 1$. For these reasons we define the following objects. 

Fix an integer  $r\ge 1$ and fix $a_1,\dots,a_r$ nonnegative real numbers. Let $f(x)=(x+a_1)(x+a_2)\cdots(x+{a_r})$. We then define for all $n,j\in\N$
\begin{align}
H^f_{j,n}&=h_{n-j}(f(1),f(2),\dots,f(j));\label{D:Hf}\\
E^f_{j,n}&=e_{n-j}(f(1),f(2),\dots,f(n-1)).\label{D:Ef}
\end{align}

By the above remarks, and  denoting with $S(n,j)$, $c(n,j)$ the (unsigned) Stirling numbers of the second and first kind, and with $LS(n,j)$, $Lc(n,j)$ the (unsigned) Legendre-Stirling numbers of the second and first kind, it is easy to check that
\begin{align}
S(n,j)&=H_{j,n}^{x};& c(n,j)&=E_{j,n}^{x} \label{E:Stir};\\
LS_n^{(j)}&=H_{j,n}^{x(x+1)};& Lc_n^{(j)}&=E_{j,n}^{x(x+1)}\label{E:LS};\\
JS_n^{(j)}(z)&=H_{j,n}^{x(x+z)};&Jc_n^{(j)}(z)&=E_{j,n}^{x(x+z)}.\label{E:JS}
\end{align}

When the polynomial $f$ is in $\N[x]$ it is possible to define a $q$-analogue of the numbers $H_{j,n}^{f}$ and $E_{j,n}^{f}$.

Given a positive integer $n$, we denote by $[n]$ the polynomial $1+q+\cdots+ q^{n-1}$; moreover, we set $[0]:=0$. 

\begin{definition}
\label{D:qanal}
Let $f(x)=(x+a_1)\cdots (x+a_r)$, with $r\in\PPP$ and $a_1,\dots,a_r\in\N$.
For all $j,n\ge1$, we define the polynomials $H_{j,n}^{f}(q)$ by the recursive formula
\begin{align}
H_{j,n}^{f}(q)=&H_{j-1,n-1}^{f}(q)+[j+a_1][j+a_2]\cdots[j+a_r]H_{j,n-1}^{f}(q);\label{A:qH}
\end{align}
with initial conditions
\begin{align}
H_{0,n}^{f}(q)=&H_{j,0}^{f}(q)=0; \quad H_{0,0}^{f}(q)=1.\notag
\end{align}
For all $j,n\ge 1$, we define the polynomials $E_{j,n}^{f}(q)$ by the recursive formula
\begin{align}
E_{j,n}^{f}(q)=&E_{j-1,n-1}^{f}(q)+[n-1+a_1][n-1+a_2]\cdots[n-1+a_r]E_{j,n-1}^{f}(q);\label{A:qE}
\end{align}
with initial conditions
\begin{align}
E_{0,n}^{f}(q)=&E_{j,0}^{f}(q)=0; \quad E_{0,0}^{f}(q)=1.\notag
\end{align}
\end{definition}

In the case of the Stirling numbers we recognize well-known $q$-analogue (see e.\ g.\ \cite{Carlitz1933}, \cite{Gould1961}, \cite{Leroux1990}).

In the following, for all $j,n,k\in\N$ we denote by $H_{j,n}^f(k)$ and $E_{j,n}^f(k)$  the coefficient of $q^k$ in $H_{j,n}^f(q)$ and $E_{j,n}^f(q)$, respectively.

In the next section we give a combinatorial interpretation for $H_{j,n}^f[k]$ and $E_{j,n}^f[k]$; in particular we obtain a combinatorial interpretation of $H_{j,n}^f$ and $E_{j,n}^f$. This combinatorial interpretation is new even for $f=x(x+z)$ (i.\ e.,  for the Jacobi-Stirling numbers).
In Section \ref{S:Int2} we consider $H_{j,n}^f$ and $E_{j,n}^f$ as polynomials in the indeterminates $a_1,\dots, a_r$ where $f(x)=(x+a_1)\cdots (x+a_r)$ and we give other combinatorial interpretations.

In Section \ref{s:Dev} we study the evaluations of the following functions which generalize the elementary and complete symmetric functions.

In \cite[Section 7]{Gessel1989} Gessel and Viennot introduce the generalized Schur functions, defined on a set $X=\{x_1,\dots, x_n\}$ of indeterminates, with an order relation $\Rel$, in the following way
\begin{equation}
\label{E:GenSchurFun}
h_k^{\Rel}(x_1,\dots,x_n):=\sum_{i_1,\dots,i_k} x_{i_1}x_{i_2}	\cdots x_{i_k}
\end{equation}
where the sum is over all indices such that $i_1 \Rel i_2 \Rel \cdots \Rel i_k$. If $k=0$ then $h_0^{\Rel}=0$.
If we consider the classical order relations $<$ or $\le$ we get the elementary or complete symmetric functions. In \cite[Section 5]{Brenti1995} the author turns his attention to the relation $\Rel_t$ defined as follows: for all $t\in\N$ and for all $i,j\in\N$, we set $x_i\Rel_t x_j$ if and only if $j-i\ge t$.
Another generalization of the elementary symmetric functions is given in \cite[Section 5]{Brenti1995} as follows: for all $t,n,k\in\PPP$ we set
\begin{equation}
\label{E:GenaFun}
a_k^{(t)}(x_1,\dots,x_n):=\sum_{i_1,\dots,i_k} x_{i_1}x_{i_2}	\cdots x_{i_k}
\end{equation}
where the sum is over all $1\le i_1<\cdots <i_k\le n$ such that $i_j\equiv j$ (mod $t$) for all $j=1,\dots,k$. For example, $a_2^{(2)}(x_1,\dots,x_6)=x_1x_2+x_1x_4+x_1x_6+x_3x_4+x_3x_6+x_5x_6$.

In Section \ref{s:Dev} we study the evaluations of the above polynomials and we have combinatorial interpretations of their coefficients.

\section{Combinatorial interpretations of $H_{j,n}^{f}[k]$ and $E_{j,n}^{f}[k]$}
\label{S:Int1}
In this section we give a combinatorial interpretation of the coefficients $H_{j,n}^f[k]$ and $E_{j,n}^f[k]$ for any fixed polynomial with only integer roots and nonnegative coefficients $f=(x+a_1)(x+a_2)\cdots(x+a_r)$, $a_1\le a_2\le\cdots \le a_r$, and for fixed $j,n,k\in\N$. 

\subsection{Combinatorial interpretation of $H_{j,n}^{f}[k]$.}
\label{SS:Int1H}

Fix $j,n,k\in \N$ and consider $r$ labeled copies of the numbers $1,\dots,n$, i.\ e.\
\begin{equation}
\label{E:copiesnumbers}
1_1,1_2,\dots,1_r,2_1,2_2,\dots,2_r,\dots\dots,n_1,n_2,\dots,n_r.
\end{equation}
We consider a pair $P=(\pi,(S_1,\dots, S_{a_r}))$ where $\pi$ is a set partition of a subset of $\{1_1,\dots, n_{r}\}$ into $j$ blocks and $S_1,\dots, S_{a_r}$ are subsets of $\{1_1,\dots, n_r\}$. We say that $P$ is  {\bf $f$-Stirling} of order $(n,j)$ if  $P$ is  a partition of (\ref{E:copiesnumbers}) into $j+a_r$ subsets such that
\begin{itemize}
	\item the subsets in $\pi$ are nonempty  and each one contains the minimum number with all its indices;
	\item one of the subsets in $\pi$  contains $1_1,1_2,\dots,1_r$;
	\item each  $m_i$, ($1\le m\le n$, $1\le i\le r$) is in one of the first $j+a_i$ subsets.  
\end{itemize}

For example, an $f$-Stirling partition of order $(3,2)$, with $f=(x+1)(x+2)$, is $(\{1_1,1_2,3_1\},\{2_1,2_2\},\{\},\{3_2\})$. The partitions $P_1=(\{1_1,1_2,2_1\},\{2_2,3_1,3_2\},\{\},\{\})$, $P_2=(\{2_1,2_2\},\{3_1,3_2\},\{1_1,1_2\},\{\}) $ and $P_3=(\{1_1,1_2,3_2\},\{2_1,2_2\},\{\},\{3_1\})$ are not $f$-Stirling since one of the previous conditions fails. 

Now, given an $f$-Stirling partition $P=(\pi,(S_1,\dots, S_{a_r}))$ of order $(n,j)$, we label each subset in $P$ in the following way: each one of the subsets in $\pi$ is labeled by the minimum number that it contains; each $S_i$ is labeled by $1-i$ for all $i=1,\dots,a_r$. For example, all sets in $(\{1_1,1_2,3_1\},\{2_1,2_2\},\{\},\{3_2\})$ have labels $(1,2,0,-1)$. With these labels, we can define a total order relation between the subsets of $P$. Moreover we  say that a subset is greater or less than a number via its label. By using these order relations we can define the following numbers.

Let $i_j$ ($i\le n$, $j\le r$) be a labeled number in (\ref{E:copiesnumbers}) and let $A_{i_j}$ be the unique subset in $P$ containing $i_j$. We set
$$
s_{i_j}:=\big\vert\{A \text{ subset in } P\vert A_{i_j}<A<i  \}\big\vert
$$
and 
$$
s_P:=\sum_{1\le i\le n, 1\le j\le r}s_{i_j}.
$$
We define $s_P:=0$ when $n=0$. For example, if $P=(\{1_1,1_2,3_1\},\{2_1,2_2\},\{\},\{3_2\})$ then $s_P=4$.

We then have the following result
\begin{proposition}
\label{P:IntGS}
For all $n,j,k\in\N$ we have that $H_{j,n}^{f}[k]$ is the number of $f$-Stirling partitions $P$ of order $(n,j)$ such that $s_P=k$.
\end{proposition}

\begin{proof}
We argue by induction on $n$. We let $p_{j,n}^{f}(k)$ be the number of $f$-Stirling partitions $P$ of order $(n,j)$ with $s_P=k$.
If $n=0$ we then  have to put no numbers in $j+a_r$ subsets such that the first $j$ subset are nonempty: this is possible in exactly 1 way only if $j=0$; indeed $H_{0,0}^{f}(q)=1$. If $j=0$ and $n>1$ we can't put $1_1,\dots,1_r$ in one of the first $0$ subsets, so there aren't any $f$-Stirling partitions of order $(n,0)$. Therefore $p_{j,n}^{f}(k)=H_{j,n}^{f}[k]$ if $n=0$ or $j=0$.

Now,  suppose $j,n\ge 1$ and the claim true for smaller values. We can obtain an $f$-Stirling partition $P$ either by adding a subset $\{n_1,\dots,n_k\}$ to an $f$-Stirling partition $P_1$ of order $(n-1,j-1)$, or by adding each number $n_i$ with $1\le i\le r$ to one of the first $j+a_i$ subset of an $f$-Stirling partition $P_2$ of order $(n-1,j)$. In the first case we have only one possibility and, $A_{n_i}$ has label $n$ for all $i=1,\dots,r$, so $s_P$=$s_{P_1}$. In the second case we have more possibilities. Fix an index $i\le r$, then we can put $n_i$ in $j+a_i$ subsets and so $s_{n_i}$ can assume all values $0,\dots,j+a_i-1$ (we can have $A<n$ for any $A$ subset of $P_2$). Therefore, we can construct $P$ with $s_P=k$ in 
\begin{equation}
\label{E:relp}
p_{j,n}^{f}(k)=p_{j-1,n-1}^{f}(k)+\sum_{h_1=0}^{j-1+a_1}\cdots\sum_{h_r=0}^{j-1+a_r}p_{j,n-1}^{f}\big(k-(h_1+\cdots +h_r)\big)
\end{equation}
different ways.
We know by definition (\ref{A:qH}), that 
$$
H_{j,n}^{f}(q)=H_{j-1,n-1}^{f}(q)+\sum_{h_1=0}^{j-1+a_1}\cdots\sum_{h_r=0}^{j-1+a_r}q^{h_1+\cdots +h_r}H_{j,n-1}^{f}(q).
$$
Therefore, the coefficient of $q^k$ is
\begin{equation}
\label{E:relHk}
H_{j,n}^{f}[k]=H_{j-1,n-1}^{f}[k]+\sum_{h_1=0}^{j-1+a_1}\cdots\sum_{h_r=0}^{j-1+a_r}H_{j,n-1}^{f}\big[k-(h_1+\cdots +h_r)\big].
\end{equation}
By comparing (\ref{E:relp}) and (\ref{E:relHk}) the thesis follows.
\end{proof}

For example here are all $f$-Stirling partitions of order $(3,2)$ when  $f=(x+1)(x+2)$: 
\begin{align*}
\text{label} && 1 	&& 2  	&& 0 						&& -1 &&\\ 
P_1: &&\{1_1,1_2,3_1,3_2\} && \{2_1,2_2\} && \{\} && \{\} && s_{P_1}=2\\
P_2: &&\{1_1,1_2,3_1\} && \{2_1,2_2,3_2\} && \{\} && \{\} && s_{P_2}=1\\
P_3: &&\{1_1,1_2,3_1\} && \{2_1,2_2\} && \{3_2\} && \{\} && s_{P_3}=3\\
P_4: &&\{1_1,1_2,3_1\} && \{2_1,2_2\} && \{\} && \{3_2\} && s_{P_4}=4\\
P_5: &&\{1_1,1_2,3_2\} && \{2_1,2_2,3_1\} && \{\} && \{\} && s_{P_5}=1\\
P_6: &&\{1_1,1_2\} && \{2_1,2_2,3_1,3_2\} && \{\} && \{\} && s_{P_6}=0\\
P_7: &&\{1_1,1_2\} && \{2_1,2_2,3_1\} && \{3_2\} && \{\} && s_{P_7}=2\\
P_8: &&\{1_1,1_2\} && \{2_1,2_2,3_1\} && \{\} && \{3_2\} && s_{P_8}=3\\
P_9: &&\{1_1,1_2,3_2\} && \{2_1,2_2\} && \{3_1\} && \{\} && s_{P_9}=3\\
P_{10}: &&\{1_1,1_2\} && \{2_1,2_2,3_2\} && \{3_1\} && \{\} && s_{P_{10}}=2\\
P_{11}: &&\{1_1,1_2\} && \{2_1,2_2\} && \{3_1,3_2\} && \{\} && s_{P_{11}}=4\\
P_{12}: &&\{1_1,1_2\} && \{2_1,2_2\} && \{3_1\} && \{3_2\} && s_{P_{12}}=5\\
\text{label} && 1 	&& 3  	&& 0 						&& -1 &&\\ 
P_{13}: &&\{1_1,1_2,2_1,2_2\} && \{3_1,3_2\} && \{\} && \{\} && s_{P_{13}}=0\\
P_{14}: &&\{1_1,1_2,2_1\} && \{3_1,3_2\} && \{2_2\} && \{\} && s_{P_{14}}=1\\
P_{15}: &&\{1_1,1_2,2_1\} && \{3_1,3_2\} && \{\} && \{2_2\} && s_{P_{15}}=2\\
P_{16}: &&\{1_1,1_2,2_2\} && \{3_1,3_2\} && \{2_1\} && \{\} && s_{P_{16}}=1\\
P_{17}: &&\{1_1,1_2\} && \{3_1,3_2\} && \{2_1,2_2\} && \{\} && s_{P_{17}}=2\\
P_{18}: &&\{1_1,1_2\} && \{3_1,3_2\} && \{2_1\} && \{2_2\} && s_{P_{18}}=3\\
\end{align*}

Indeed $H_{2,3}^{(x+1)(x+2)}(q)=2+4q+5q^2+4q^3+2q^4+q^5$.
When $r=1$ and $a_1=0$, with the same arguments of the proof of Proposition \ref{P:IntGS}, we have the following new combinatorial interpretation of $q$-Stirling numbers of the second kind.
\begin{corollary}
\label{C:Stirling2}
For all $j,n,k\in\N$, $S(n,j)[k]$ is the number of  partitions $P$ of $\{1,\dots,n\}$ in $j$ nonempty blocks with $s_P=k$.
\end{corollary}

When we evaluate all polynomials in $q=1$ we get
\begin{corollary}
\label{C:interH}
For all $j,n\in\N$, $H_{j,n}^{f}$ is the number of $f$-Stirling partitions of order $(n,j)$.
\end{corollary}
In particular, when $r=2$, $a_1=0$ and $a_2=1$ (i.\ e.\ for the Legendre-Stirling numbers), we obtain The following result.
\begin{corollary}
For all $n,j,k\in\N$ we have that $LS_n^{(j)}[k]$ is the number of $x(x+1)$-Stirling partitions $P$ of order $(n,j)$ such that $S_P=k$.
\end{corollary}
 This interpretation is similar to one in  \cite[Theorem 2]{Andrews2009}. Here  the authors fill $j+1$ sets with the numbers in (\ref{E:copiesnumbers}), such that there exists a set (the "zero box") which is the only set that may be empty and it may not contain both copies of any number; the other $j$ sets are indistinguishable and each is non-empty; each such set contains both copies of its smallest element and does not contain both copies of any other elements. Consider an $x(x+1)$-Stirling partition and for all $m\le n$ if $A_{m_1}\le A_{m_2}$ move the element $m_1$ in the set immediately lower than $A_{m_1}$. It gives a bijective proof of the equivalence of both interpretations.

Finally, for the Jacobi-Stirling numbers,  Corollary \ref{C:interH} becomes:
\begin{corollary}

For all $j,n,k,z\in \N$, $JS_n^{(j)}[k]$ is the number of $x(x+z)$-Stirling partitions of order $(n,j)$.
\end{corollary}

\subsection{Combinatorial interpretation of $E_{j,n}^{f}[k]$}
\label{SS:Int1E}

In \cite{Egge2010}, for any cycle of a permutation Egge defines the \emph{cycle maxima} (resp. \emph{cycle minima}) as the maximum (resp. minimum) of the numbers in the cycle. 
We consider an $r$-tuple of permutations $\pi=(\pi_1,\dots,\pi_r)$ with $\pi_i\in S_{n+a_i}$ for all $i=1,\dots,r$ and 	we say that $\pi$ is {\bf $f$-Stirling} of order $(n,j)$ if and only if 
\begin{enumerate}[(a)]
  \item each $\pi_i$ has exactly $j+a_i$ cycles;
	\item $\pi_1,\pi_2,\dots,\pi_r$ have the same cycle maxima less than $n+1$;
	\item the orbits of $n,n+1,\dots,n+a_i$ in $\pi_i$ are pairwise distinct for all $i\le r$.
\end{enumerate}

For example, an $f$-Stirling $2$-tuple permutation of order $(3,2)$, with $f=(x+1)(x+2)$, is $\{(4)(3,1)(2), (5)(4,1)(3)(2)\}$. The $2$-tuples $\sigma_1=\{(4)(3,1)(2), (5)(4,1)(2,3)\}$, $\sigma_2=\{(4)(3,1)(2), (5)(4,2)(3)(1)\}$ and $\sigma_3=\{(4)(3,1)(2), (5,4)(3)(2)(1)\}$ are not $f$-Stirling since one of the previous conditions fails. 

Given a permutation $\rho\in S_n$ we define a word  in the alphabet $\{1,\dots, n\}$ in this way: we write each cycle of $\rho$ with the cycle maxima in the first place and we order the cycles by their cycle maxima in  decreasing order; then we omit the brackets. We call $s_\rho$ this word. For example, if $\rho=(162)(45)\in S_6$ then $s_\rho=621543$; if $\sigma=(12)(45)\in S_6$ then $s_\sigma=654321$. Note that we can obtain the same word from two different permutations: for example if $\rho_1=(12)(45)\in S_6$, $\rho_2=(132)(654)\in S_6$ then $s_{\rho_1}=s_{\rho_2}=654321$.
Now let $s$ be a such word in the alphabet $\{1,\dots,n\}$. Let's define 
$$
\coinv_s^{(i)}:=\Big\vert\big\{j\in\{1,\dots,n\}\vert j>i \text{ and $i$ is on the left of }j  \text{ in } s \big\}\Big\vert.
$$
For example, if $s=623541$ then $\coinv_s^{(2)}=3;\coinv_s^{(5)}=0$. We then  set
$$
\coinv_s:=\sum_{i=1}^{n}\coinv_s^{(i)}.
$$
We call \emph{coinversions} of $s$ the number $\coinv_s$.
Equivalently, $\coinv_{s}$ is the minimum number of exchanges of two consecutive elements in $s$ to obtain the word $n,n-1,\dots, 2,1$.

\begin {proposition}
\label{P:Int1}
Fix $n,j,k\in \N$. Then $E_{j,n}^{f}[k]$ is the number of $f$-Stirling $r$-tuples of permutations $\pi=(\pi_1,\dots,\pi_r)$ such that 
\begin{equation}
\label{E:CondkGc}
\sum_{i=1}^r \coinv_{\pi_i}=k.
\end{equation}
\end{proposition}

\begin{proof}
We proceed by induction on $j$. If $j=0$, by conditions (a) and (c), we have only one possibility when $n=0$. If $n<j$, condition (a) is never satisfied. Indeed in both cases $E_{j,n}^{f}(q)=0$ except $E_{0,0}^{f}(q)=1$. Therefore, let's suppose $1\le j\le n$. We denote by $p_{j,n}^{f}(k)$ the number of $f$-Stirling $r$-tuples of permutations $\pi$ satisfying (\ref{E:CondkGc}).

We can make an $f$-Stirling $r$-tuple of permutations of order $(n,j)$ in two distinct ways:
1) start from an $f$-Stirling $r$-tuple of permutations $\pi$ of order $(n-1,j-1)$, increase each number of one unit and then add in each permutation $\pi_i$ (with abuse of notations) the trivial cycle $(1)$. In this case, for all $i\le r$ $\coinv_{\pi_i}^{(1)}=0$.  \\
2) start from an $f$-Stirling $r$-tuple of permutations $\pi$ of order $(n-1,j)$, increase each number of one unit and then add the element $1$ in each permutation $\pi_i$. This operation can be done in $n-1+a_i$ ways for all permutation $\pi_i$ (it is equivalent to put the new number in the word $s_{\pi_i}$ in all positions except the first) and $\coinv_{\pi_i}^{(1)}$ can assume all values between $0$ and $n-2+a_i$.
Therefore, we have 
\begin{equation}
\label{E:relpE}
p_{j,n}^{f}(k)=p_{j-1,n-1}^{f}(k)+\sum_{h_1=0}^{n-2+a_1}\cdots\sum_{h_r=0}^{n-2+a_r}p_{j,n-1}^{f}\Big(k-(h_1+\cdots +h_r)\Big).
\end{equation}
By (\ref{A:qE}), we know that
$$
E_{j,n}^{f}(q)=E_{j-1,n-1}^{f}(q)+\sum_{h_1=0}^{n-2+a_1}\cdots\sum_{h_r=0}^{n-2+a_r}q^{h_1+\cdots +h_r}E_{j,n-1}^{f}(q).
$$
Therefore, the coefficient of $q^k$ is
\begin{equation}
\label{E:relEk}
E_{j,n}^{f}[k]=E_{j-1,n-1}^{f}[k]+\sum_{h_1=0}^{n-2+a_1}\cdots\sum_{h_r=0}^{n-2+a_r}E_{j,n-1}^{f}\Big[k-(h_1+\cdots +h_r)\Big].
\end{equation}

By comparing (\ref{E:relpE}) and (\ref{E:relEk}) the thesis follows.
\end{proof}

For example, here are all $f$-Stirling $2$-tuples of permutations of order $(3,2)$ when $f=(x+1)(x+2)$. 
\begin{align*}
\sigma_1: &&(4,1)(3)(2); && (5,1)(4)(3)(2) && \coinv=2+3=5 \\
\sigma_2: &&(4,1)(3)(2); && (5)(4,1)(3)(2) && \coinv=2+2=4 \\
\sigma_3: &&(4,1)(3)(2); && (5)(4)(3,1)(2) && \coinv=2+1=3 \\
\sigma_4: &&(4,1)(3)(2); && (5)(4)(3)(2,1) && \coinv=2+0=2 \\
\sigma_5: &&(4)(3,1)(2); && (5,1)(4)(3)(2) && \coinv=1+3=4 \\
\sigma_6: &&(4)(3,1)(2); && (5)(4,1)(3)(2) && \coinv=1+2=3 \\
\sigma_7: &&(4)(3,1)(2); && (5)(4)(3,1)(2) && \coinv=1+1=2 \\
\sigma_8: &&(4)(3,1)(2); && (5)(4)(3)(2,1) && \coinv=1+0=1 \\
\sigma_9: &&(4)(3)(2,1); && (5,1)(4)(3)(2) && \coinv=0+3=3 \\
\sigma_{10}: &&(4)(3)(2,1); && (5)(4,1)(3)(2) && \coinv=0+2=2 \\
\sigma_{11}: &&(4)(3)(2,1); && (5)(4)(3,1)(2) && \coinv=0+1=1 \\
\sigma_{12}: &&(4)(3)(2,1); && (5)(4)(3)(2,1) && \coinv=0+0=0 \\
\sigma_{13}: &&(4,2)(3)(1); && (5,2)(4)(3)(1) && \coinv=1+2=3 \\
\sigma_{14}: &&(4,2)(3)(1); && (5)(4,2)(3)(1) && \coinv=1+1=2 \\
\sigma_{15}: &&(4,2)(3)(1); && (5)(4)(3,2)(1) && \coinv=1+0=1 \\
\sigma_{16}: &&(4)(3,2)(1); && (5,2)(4)(3)(1) && \coinv=0+2=2 \\
\sigma_{17}: &&(4)(3,2)(1); && (5)(4,2)(3)(1) && \coinv=0+1=1 \\
\sigma_{18}: &&(4)(3,2)(1); && (5)(4)(3,2)(1) && \coinv=0+0=0
\end{align*}
Indeed, $E_{2,3}^{(x+1)(x+2)}(q)=2+4q+5q^2+4q^3+2q^4+q^5$.
When $r=1$ and $a_1=0$, with the same arguments of the proof of Proposition \ref{P:Int1}, we have the following new combinatorial interpretation of $q$-Stirling numbers of the first kind.
\begin{corollary}
\label{C:Stirling1}
For all $j,n,k\in\N$, $S(n,j)[k]$ is the number of  permutations $\rho\in S_n$ with $j$ cycles  with $\coinv_{s_\rho}=k$.
\end{corollary}

When we evaluate all polynomials in $q=1$ we have
\begin{corollary}
\label{C:intE}
For all $j,n\in\N$, $E_{j,n}^{f}$ is the number of $f$-Stirling $r$-tuples of permutations of order $(n,j)$.
\end{corollary}

For $r=2$, $a_1=0$ and $a_2=1$, Corollary \ref{C:intE} reduces to \cite[Theorem 2.5]{Egge2010}. 

\section{Combinatorial interpretations as polynomials}
\label{S:Int2}
In this section we consider the numbers $H_{j,n}^{f}$ and $E_{j,n}^{f}$ as polynomials in  $a_1,\dots,a_r$ with $r$  a fixed positive integer and $f(x)=(x+a_1)\cdots(x+a_r)$.
\subsection{Combinatorial interpretation of $H_{j,n}^{f}$}
\label{SS:Int2H}
Fix $j,n\in \N$ and consider exactly $r$ copies of numbers $0,\dots,n$, i.\ e.\
\begin{equation}
\label{E:copiesnumbers2}
0_1,0_2,\dots,0_r,1_1,1_2,\dots,1_r,\dots\dots,{n}_1,{n}_2,\dots,{n}_r.
\end{equation}
We consider a partition $P$ of (\ref{E:copiesnumbers2}) and we say that $P$ is  \emph{$r$-Stirling} of order $(n,j)$ if  $P$ is  a partition of (\ref{E:copiesnumbers2}) into $j+1$ subsets such that
\begin{itemize}
	\item all the $j+1$ subsets are nonempty;
	\item each subset contains the minimum number with all its indices;
	\item if $n\ne 0$ then no subset contains both $0_i,1_i$ for any index $i\le r$. 
\end{itemize}
Given such a partition $P$, we call $0$-subset of $P$ the only subset that contains $0_1,\dots,0_r$.

For example, a $2$-Stirling partition of order $(3,2)$ is $(\{0_1,0_2,3_2\}, \{1_1,1_2,3_1\},$ $\{2_1,2_2\})$. The partitions $P_1=(\{0_1,0_2,2_1,2_2\},\{1_1,1_2,3_1,3_2\}, \{\})$, $P_2=(\{0_1,0_2,1_1\},$ $\{1_2,2_1,2_2\}, \{3_1,3_2\})$ and $P_3=(\{0_1,0_2,1_1,1_2\}, \{2_1,2_2\}, \{3_1,3_2\})$ are not $2$-Stirling since one of the previous condition fails.

\begin{proposition}
\label{P:Int2H}
For all $j,n\in\N$ and for all $\beta_1,\dots,\beta_r\in\N$ we have that the coefficient of $a_1^{\beta_1}a_2^{\beta_2}\cdots a_r^{\beta_r}$ in $H_{j,n}^f$ is the number of $r$-Stirling partitions of order $(n,j)$ whose  $0$-subset contains $\beta_i+1$ numbers with index $i$ for all $i\in\{1,\dots,r\}$. 
\end{proposition}
\begin{proof}
We argue by induction on $n$. We set $p_{j,n}(\beta_1,\dots,\beta_r)$ the number of $r$-Stirling partitions $P$ of order $(n,j)$  such that the $0$-subset of $P$ contains $\beta_i+1$ numbers with index $i$ for all $i\in\{1,\dots,r\}$. If $n=0$ then we put the numbers $0_1,\dots,0_r$ in the same subset. Therefore there exists $r$-Stirling partitions of order $(0,j)$ if and only if $j=0$. Indeed, we have $H_{0,j}^f=\delta_{0,j}$.
If $j=0$, since we cannot put $0_1,\dots,0_r,1_1,\dots,1_r$ in the same subset, then necessarily $n=0$. 

Now, let we suppose that $j,n\ge 1$ and the claim true for smaller values. We can obtain an $r$-Stirling partition $P$ or by adding a subset $\{n_1,\dots,n_r\}$ to an $r$-Stirling partition $P_1$ of order $(n-1,j-1)$, or by adding each number $n_i$ with $1\le i \le r$ in one of the $j+1$ subsets of an $r$-Stirling partition $P_2$ of order $(n-1,j)$. In the first case, the $0$-subset of $P_1$ is the same of the $0$-subset of $P$. In the second case, for all $i\in\{1,\dots,r\}$ we have $j+1$ possibility to fix the position of $n_i$, only one of which changes the $0$-subset.

Therefore, we can obtain a partition $P$ as in the claim in
\begin{align}
p_{j,n}&(\beta_1,\dots,\beta_r)=p_{j-1,n-1}(\beta_1,\dots,\beta_r)+\notag\\
&+\sum_{i_1=0}^{\min(\beta_1,1)}\cdots \sum_{i_r=0}^{\min(\beta_r,1)} j^{r-(i_1+\cdots + i_r)} p_{j,n-1}(\beta_1-i_1,\dots,\beta_r-i_r)\label{E:NumPH}
\end{align} 
different ways. Easy to check, by recurrence (\ref{E:Rich}), that the coefficient of $a_1^{\beta_1}\cdots a_r^{\beta_r}$ in $H_{j,n}^f$ satisfy the same recurrence in (\ref{E:NumPH}), therefore the thesis follows.
\end{proof}
For example, here are all $2$-Stirling partitions of order $(3,2)$. 
\begin{align*}
P_1: &&\{0_1,0_2,3_1,3_2\} && \{1_1,1_2\} && \{2_1,2_2\} \\
P_2: &&\{0_1,0_2,3_1\} && \{1_1,1_2,3_2\} && \{2_1,2_2\} \\
P_3: &&\{0_1,0_2,3_1\} && \{1_1,1_2\} && \{2_1,2_2,3_2\} \\
P_4: &&\{0_1,0_2,3_2\} && \{1_1,1_2,3_1\} && \{2_1,2_2\} \\
P_5: &&\{0_1,0_2\} && \{1_1,1_2,3_1,3_2\} && \{2_1,2_2\} \\
P_6: &&\{0_1,0_2\} && \{1_1,1_2,3_1\} && \{2_1,2_2,3_2\} \\
P_7: &&\{0_1,0_2,3_2\} && \{1_1,1_2\} && \{2_1,2_2,3_1\} \\
P_8: &&\{0_1,0_2\} && \{1_1,1_2,3_2\} && \{2_1,2_2,3_1\} \\
P_9: &&\{0_1,0_2\} && \{1_1,1_2\} && \{2_1,2_2,3_1,3_2\} \\
P_{10}: &&\{0_1,0_2,2_1,2_2\} && \{1_1,1_2\} && \{3_1,3_2\} \\
P_{11}: &&\{0_1,0_2,2_1\} && \{1_1,1_2,2_2\} && \{3_1,3_2\} \\
P_{12}: &&\{0_1,0_2,2_2\} && \{1_1,1_2,2_1\} && \{3_1,3_2\} \\
P_{13}: &&\{0_1,0_2\} && \{1_1,1_2,2_1,2_2\} && \{3_1,3_2\} 
\end{align*}

Indeed, $H_{2,3}^f=5+3(a_1+a_2)+2a_1a_2$.

In the case of Jacobi-Stirling numbers we have $f=x(x+z)$: then, in our notation, we have $r=2$, $a_1=0$, and $a_2=z$. To set $a_1=0$ is equivalent to take only the monomials without factor $a_1$. By Proposition \ref{P:Int2H}, we have to consider only the $2$-Stirling partitions $P$ of order $(n,j)$ of $0_1,0_2,\dots,n_1,n_2$ with the $0$-subset without numbers labeled by $1$ except $0_1$. With this remark our interpretation Proposition \ref{P:Int2H} reduces to
\begin{corollary}
For all $j,n,\beta\in N$ the coefficient of $z^\beta$ in $JS_{j,n}(z)$ is the number of $2$-Stirling partitions of order $(n,j)$ whose $0$-subset contains $\beta+1$ numbers with index $2$ and only one number with index $1$ (necessarily $0_1$).
\end{corollary}
This result is equivalent of the one given in \cite[Theorem 2]{Gelineau2010}. Here the authors consider partitions of (\ref{E:copiesnumbers2}) such that  each set contains both copies of its smallest element and does not contain both copies of any other number. Consider  a $2$-Stirling partition whose $0$-subset has no positive numbers labeled by $1$. If $m_1,m_2$ are in the same set, with $1\le m\le n$, then move $m_1$ in the $0$-subset. This prove that the two interpretations are equivalent.

\subsection{Combinatorial interpretation of $E_{j,n}^{f}$}
\label{SS:Int2E}
Fix $j,n\in\N$ and consider an $r$-tuple of permutations $\pi=(\pi_1,\dots,\pi_r)\in S_{n+1}^r$. We say that $\pi$ is \emph{$r$-Stirling} of order $(n,j)$ if and only if 

\begin{enumerate}[(a')]
  \item each $\pi_i$ has exactly $j+1$ cycles;
	\item $\pi_1,\pi_2,\dots,\pi_r$ have the same cycle minima;
	\item if $n\ne0$ the orbits of $1$ and $2$ are disjoint for all permutations $\pi_i$.
\end{enumerate}

For example, a $2$-Stirling $2$-tuple of permutations of order $(3,2)$  is $\{(1)(2,4)(3)$, $(1)(2)(3,4)\}$. The $2$-tuples $\sigma_1=\{(1)(2,4)(3), (1)(2,3)(4)\}$ and $\sigma_2=\{(1,2)(3)(4)$, $(1,2)(3)(4)\}$ are not $2$-Stirling since one of the previous condition fails.

Following notation in \cite{Gelineau2010}, given a word $w=w(1)\dots w(l)$ on the finite alphabet $\{1,\dots,n+1\}$, a letter $w(j)$ is a \emph{record} of $w$ if $w(k)>w(j)$ for every $k\in\{1,\dots,j-1\}$. We define $\rec(w)$ to be the number of records of $w$ and given $\pi\in S_{n+1}$ we define $\rec(\sigma)=\rec(\sigma(1),\sigma^2(1),\dots,1)$ (the elements are only in the orbit of $1$).

\begin{proposition}
\label{P:Int2E}
For all $n,j\in\N$ and for all $\beta_1,\dots,\beta_r\in\N$ the coefficient of $a_1^{\beta_1}a_2^{\beta_2}\cdots a_r^{\beta_r}$ in $E_{j,n}^r$ is the number of $r$-Stirling $r$-tuples of permutations $\pi=(\pi_1,\dots,\pi_r)$ such that $\rec(\pi_i)=\beta_i+1$ for all $i=1,\dots,r$.
\end{proposition}

\begin{proof}
If $j=0$, by conditions (a') and (c'), we have only one $r$-tuple of $r$-Stirling permutations of order $(n,0)$ when $n=0$ and no one when $n>0$. If $n<j$, condition (a') is never satisfied. Indeed, in both cases $E_{j,n}^f=0$ except $E_{0,0}=1$.
Let's now suppose that $1\le j\le n$. We call $p_{j,n}^r(\beta_1,\dots,\beta_r)$ the number of $r$-Stirling $r$-tuples of permutations $\pi=(\pi_1,\dots,\pi_r)$ of order $(n,j)$ such that $\rec(\pi_i)=\beta_i+1$ for all $i=1,\dots,r$. We can construct a such $\pi$ in two different ways (and the reader can check that in these ways we obtain all such $r$-tuples). The first possibility is to start with an $r$-Stirling $r$-tuple of permutations of order $(n-1,j-1)$ and add a trivial cycle $(n+1)$ in each its permutation. In this case we preserve the orbit of $(1)$, i.\ e.\ the numbers of records of all permutations does not change. 

The second possibility is to start with an $r$-Stirling $r$-tuple $\pi'=(\pi'_1,\dots,\pi'_r)$ of permutations of order $(n-1,j)$ and add the number $n+1$ in each permutation $\pi'_i$. In this case we increase the numbers of records (i.\ e.\ $\rec(\pi_i)=\rec(\pi'_i)+1$) if and only if we put $n+1$ on the right of $1$, i.\ e.\ $\pi_i(1)=n+1$ (and in this case $\pi_i(n+1)$ is the new record): in fact in all other cases (exactly $n-1$ cases), the orbit of $1$ does not change or $n+1$ is in the orbit of $1$ but $\pi^{-1}_i(n+1)$ is on the left of $n+1$ and therefore $n+1$ is not a record.

So we have 
\begin{align}
p_{j,n}^r&(\beta_1,\dots,\beta_r)=p_{j-1,n-1}^r(\beta_1,\dots,\beta_r) + \notag \\   
&+ \sum_{i_1=0}^{\min(\beta_1,1)}\cdots\sum_{i_r=0}^{\min(\beta_r,1)}(n-1)^{r-(i_1+\cdots+i_r)}p_{j,n-1}(\beta_1-i_1,\cdots,\beta_r-i_r)\label{E:Int2pE}
\end{align}
different $r$-Stirling $r$-tuples of permutations of order $(n-j)$. Easy to check by recurrence relation (\ref{E:Rice}) that the coefficients of $a_1^{\beta_1}\cdots a_r^{\beta_r}$ in $E_{j,n}^f$ satisfy the same recurrence in (\ref{E:Int2pE}), therefore the thesis follows.
\end{proof}
For example, here are all $2$-Stirling $2$-tuples of permutations of order $(3,2)$. 
\begin{align*}
\sigma_1: &&(1,4)(2)(3); && (1,4)(2)(3) && \rec=(2,2) \\
\sigma_2: &&(1,4)(2)(3); && (1)(2,4)(3) && \rec=(2,1) \\
\sigma_3: &&(1,4)(2)(3); && (1)(2)(3,4) && \rec=(2,1) \\
\sigma_4: &&(1)(2,4)(3); && (1,4)(2)(3) && \rec=(1,2) \\
\sigma_5: &&(1)(2,4)(3); && (1)(2,4)(3) && \rec=(1,1) \\
\sigma_6: &&(1)(2,4)(3); && (1)(2)(3,4) && \rec=(1,1) \\
\sigma_7: &&(1)(2)(3,4); && (1,4)(2)(3) && \rec=(1,2) \\
\sigma_8: &&(1)(2)(3,4); && (1)(2,4)(3) && \rec=(1,1) \\
\sigma_9: &&(1)(2)(3,4); && (1)(2)(3,4) && \rec=(1,1) \\
\sigma_{10}: &&(1,3)(2)(4); && (1,3)(2)(4) && \rec=(2,2) \\
\sigma_{11}: &&(1,3)(2)(4); && (1)(2,3)(4) && \rec=(2,1) \\
\sigma_{12}: &&(1)(2,3)(4); && (1,3)(2)(4) && \rec=(1,2) \\
\sigma_{13}: &&(1)(2,3)(4); && (1)(2,3)(4) && \rec=(1,1) \\
\end{align*}
We have $5$ couples of permutations with records $(1,1)$, $3$ with records $(2,1)$ and $(1,2)$ and $2$ with records $(2,2)$.
The reader can check the proposition in this case since $E_{2,3}^f=5+3(a_1+a_2)+2a_1a_2$.

In the case of Jacobi-Stirling numbers we have $f=x(x+z)$: then, in our notation, we have $r=2$, $a_1=0$, and $a_2=z$. To set $a_1=0$ is equivalent to take only the monomials without factor $a_1$. By Proposition \ref{P:Int2E}, we have to consider only the $2$-Stirling $2$-tuples of permutations $\pi=(\pi_1,\pi_2)$ with the orbit of $1$ in $\pi_1$ be trivial. Therefore, Proposition \ref{P:Int2E} reduces to 

\begin{corollary}
For all $n,j,\beta\in \N$ the coefficient of $z^\beta$ in $Jc_{j,n}(z)$ is the number of $2$-Stirling $2$-permutations $(\pi_1,\pi_2)$ such that $\pi_1(1)=1$ and $\rec(\pi_2)=\beta+1$.
\end{corollary}

Obviously, $\pi_1$ can be identified with a permutation of $S_n$. With this identification, our interpretation is the same of the one given in \cite[Theroem 7]{Gelineau2010}.

\section{Developments on some generalization of symmetric functions}
\label{s:Dev}
In this section we analyze functions already defined in Section 2 and introduced in \cite{Brenti1995} which generalize the elementary and complete symmetric functions. We give a combinatorial interpretation  if we evaluate them in $f(1),f(2),\dots$  as done in the previous sections. In particular, this approach is used to obtain a new interpretation in the case of the Jacobi-Stirling numbers.

In the next two subsections we give combinatorial interpretations of \\ $h_k^{\Rel_t}(f(1),\dots,f(n))$ and $a_k^{(t)}(f(1),\dots,f(n))$, where $f$ is the polynomial with all real zeros and nonnegative coefficients.

\subsection{Combinatorial interpretation of $h_n^{\Rel_t}$}
\label{SS:HRel}
 Given $r\in\N$, $a_1,\dots,a_r\in\N$, let $f$ be the polynomial $f(x)=(x+a_1)\cdots(x+a_r)$. Let $n\in\N$ and $\sigma\in S_n$ be a permutation. We decompose $\sigma$ in disjoint cycles and we write each cycle with its minimum at the first place. Then we order and label all cycles via its minimum. Therefore, examples of decompositions are $(136)(25)(4)$, $(1652)(34)$. We say that  two or more permutations have the same \emph{ordered cycle structure} if in the previous notation, the sequence of lengths of the cycles are equals (in the following we denote by $l(c)$ the length of a cycle $c$). For example, $(136)(25)(4)$ and $(123)(45)(6)$ have the same ordered cycle structure, while $(136)(25)(4)$ and $(15)(234)(6)$ no.

We say that one or more permutations  have the same  ordered cycle structure up to $k$ if the first $k$ elements of the sequences of the lengths of the cycles are equals for all permutations and the  lengths of remaining cycles (if they exist) are equal to $1$. For example $(132)(45)$ and $(176)(24)(3)(5) $ have the same ordered cycle structure up to $2$, $(132)(45)$ and $(176)(24)(35)$ no.

Now fix a permutation $\sigma$ with ordered cycles $c_{i_1},\dots,c_{i_r}$ for some integer $r\in\N$. Each cycle is labeled by its minimum element. We define a distance between two cycles via the following definition
\begin{equation}
\label{E:DefDist}
d(c_{i},c_j):=\big\vert\{c_k\in\sigma\vert i<k\le j\}\big\vert,
\end{equation}
if $i\le j$, else $d(c_i,c_j):=d(c_j,c_i)$. It is obvious that if two permutations have the same ordered cycle structure, then the distances between corresponding disjoint cycles are the same. Now we can show the following result.

\begin{proposition}
\label{P:IntHRel}
Let $r\in\PPP$, $a_1,\dots,a_r\in\N$ and $f(x)=(x+a_1)\cdots(x+a_r)$. Then for all $n,k\in\N$ and $t\ge 2$, $h_n^{\Rel_t}(f(1),\dots,f(k))$ is the number of elements in $S_{k+1+a_1}\times\cdots\times S_{k+1+a_r}$ such that all permutations have the same ordered cycle structure up to $k+1-n$, each one with  $k+1-n+a_i$ cycles and length at most $2$ and if $c_i,c_j$ are two cycles with $l(c_i)=l(c_j)=2$ then $d(c_i,c_j)\ge t-1$.
\end{proposition}

\begin{proof}
If $n=0$ then each permutation in $S_{k+1+a_i}$ as in the statement has $k+1+a_i$ disjoint cycles. Therefore it is the trivial permutation, indeed $h_0^{\Rel_t}(f(1),\dots,f(k))=1$. Moreover, if $k<n$ then $k+1-n\le 0$ and therefore, since the permutations have the same ordered cycle structure up to $k+1-n$, all cycles are trivial; but in this case $k+1-n+a_i=k+1+a_i$ and this is impossible. Indeed, in this case $h_n^{\Rel_t}(f(1),\dots,f(t))=0$.

Let's  suppose that $n\ge 1$, $k\ge n$ and the thesis true for smaller values of $n$. We fix an $r$-tuple of permutations as in the statement and we turn our attention to the element $1$ in each permutation. By assumption, all the first cycles have the same length, in particular they are all trivial or they have length $2$. In the first case we delete the trivial cycle $(1)$ in each permutation and decrease all other elements by one. Therefore we have elements in $S_{k+a_1}\times\cdots\times S_{k+a_r}$ whose permutations have the same ordered cycle structure up to $k-n$ and $k-n+a_i$ cycles. By induction, their number is $h_n^{\Rel_t}(f(1),\dots,f(k-1))$. In the second case, the first cycle of the $i$-th permutation contains a number in $\{2,\dots, k+1+a_i\}$. We delete all the first cycles (there are $\prod_{i=1}^r (k+a_i)=f(k)$ different first cycles) and rename all remaining numbers preserving the natural order of them. We get permutations in $S_{k-1+a_i}$ with $k-n+a_i$ cycles. Moreover, by assumption in each permutation the first $t-2$ cycles are trivial (check the distances) and then we can delete them. We get therefore permutations in $S_{(k-t)+1+a_i}$ with $(k-t)-(n-1)+1+a_i$ cycles. Their number is by induction $h_{n-1}^{\Rel_t}(f(1),\dots,f(k-t))$. 

By definition, it is simple to check that for all $n,k\ge 1$
\begin{equation}
\label{E:RecursionHRel}
h_n^{\Rel_t}(x_1,\dots,x_k)=h_n^{\Rel_t}(x_1,\dots,x_{k-1})+x_k h_{n-1}^{\Rel_t}(x_1,\dots,x_{k-t}).
\end{equation}
This complete the proof.
\end{proof}

As example, let $f(x)=x(x+1)$, $t=2$, $k=3$, $n=2$. It is simple to check that $h_2^{\Rel_2}(f(1),f(2),f(3))=f(1)f(3)=24$. We want obtain elements in $S_4\times S_5$ as in the previous statement. The first permutation will be $(12)(34)$, $(13)(24)$ or $(14)(23)$; the second permutation, which has the same ordered cycle structure  up to $2$ of the first, will be one of $(12)(34)(5)$, $(12)(35)(4)$, $(13)(24)(5)$, $(13)(25)(4)$, $(14)(23)(5)$, $(14)(25)(3)$, $(15)(23)(4)$, $(15)(24)(3)$. Indeed, we have $24$ possibilities.

If $f(x)=x(x+1)$, $t=3$, $k=4$, $n=2$, then $h_2^{\Rel_3}(f(1),f(2),f(3),f(4))=f(1)f(4)=40$. In this case we want elements in $S_5\times S_6$ as in the statement. The first permutation is one between $(12)(3)(45)$, $(13)(2)(45)$, $(14)(2)(35)$ and $(15)(2)(34)$; the second permutation is one between $(12)(3)(45)(6)$, $(12)(3)(46)(5)$, $(13)(2)(45)(6)$, $(13)(2)(46)(5)$, $(14)(2)(35)(6)$, $(14)(2)(36)(5)$, $(15)(2)(34)(6)$,\\ $(15)(2)(36)(4)$, $(16)(2)(34)(5)$, $(16)(2)(35)(4)$. We have $40$ possibilities.

In \cite{Brenti1995} we find the polynomials $\bar h_n^{\Rel_t}(x_1,\dots,x_k)=\sum_{i_1,\dots,i_r} x_{i_1}\cdots x_{i_r}$, that are the same as $h_n^{\Rel_t}$ but with another condition $i_1\ge t$. In this case, the interpretation is the same as that in Proposition \ref{P:IntHRel} but the ordered cycle structure is up to $k-n-t+2$. The proof is essentially the same.

It is possible to consider $h_n^{\Rel_t}(f(1),\dots,f(k))$ as polynomial in  $a_1,\dots, a_r$.
For this purpose we introduce the following definition. Given a permutation  $\sigma$ we say  that a number $n$ is a \emph{big number} if in the expansion of $\sigma$ in ordered disjoint cycles, $n$ is greater than each number on its right. For example in $(15)(24)(3)$ the big numbers are $5,4,3$; in $(14)(25)(3)$ the big numbers are $5,3$.  

\begin{proposition}
\label{P:IntHRelCoeff}
Let $r\in\PPP$,$f(x)=(x+a_1)\cdots(x+a_r)$  and $\beta_1,\dots, \beta_r\in\N$. Then for all $n,k\in\N$ and $t\ge 2$,  the coefficient of $a_1^{\beta_1}\cdots a_r^{\beta_r}$ in $h_n^{\Rel_t}(f(1),\dots,f(k))$ is the number of elements in $S_{k+2}^r$ such that all permutations have the same ordered cycle structure up to $k+1-n$, each one with  $k+2-n$ cycles and length at most $2$; if $c_i,c_j$ are two cycles with $l(c_i)=l(c_j)=2$ then $d(c_i,c_j)\ge t-1$ and for all $i=1,\dots, r$, the $i$-th permutation has $\beta_i+1$ big numbers.
\end{proposition}
The proof is essentially the same as the one of Proposition \ref{P:IntHRel}. The only difference is  when all the first cycles have lengths $2$. When the maximum number is in the first cycle of the $i$-th permutation (so this number is big), we have a contribution of the indeterminate $a_i$. 

\subsection{Combinatorial interpretation of $a_n^{(t)}$}
\label{SS:ARel}
In this subsection we give a combinatorial interpretation of $a_n^{(t)}(f(1),\dots,f(k))$, where $f$ is as usual. For permutations, we use the same notation introduced in the previous subsection.

\begin{proposition}
\label{P:IntA}
Let $r\in\PPP$, $a_1,\dots, a_n\in\N$ and $f(x)=(x+a_1)\cdots (x+a_r)$. Then for all $n,k\in\N$ and $t\ge1$ $a_n^{(t)}(f(1),\dots,f(k))$ is the number of elements in $S_{k+1+a_1}\times\cdots\times S_{k+1+a_r}$ such that all permutations have the same ordered cycle structure up to $k+1-n$, each one with  $k+1-n+a_i$ cycles and if the $j$-th cycle is not trivial then $k+1-n\equiv j$ mod $t$.
\end{proposition}

\begin{proof}
If $n=0$ then each permutation in $S_{k+1+a_i}$ as in the statement has $k+1+a_i$ disjoint cycles, then it is the trivial permutation, indeed $a_0^{t}(f(1),\dots,f(k))=1$. Moreover, if $k<n$ then $k+1-n\le 0$ and therefore, by hypothesis of ordered cycle structure up to $k+1-n$, all cycles are trivial; but in this case $k+1-n+a_i=k+1+a_i$ and this is impossible. Indeed, in this case $a_n^{(t)}(f(1),\dots,f(k))=0$.

Now let $n\ge 1$, $k\ge n$ and the thesis true for smaller values of $n$. We fix an $r$-tuple of permutation as in the statement and we analyze the cycles labeled by $1$. If $k\not\equiv n$ mod $t$ then they are all trivial, and we delete them. In this way we may obtain  $a_n^{(t)}(f(1),\dots,f(k-1))$ different $r$-tuples by induction. If $k\equiv n $ mod $t$, then the cycles labeled by $1$ have  arbitrary lengths (but all the same). Let we assume that these lengths are equal to $h+1$ with $h\ge0$. Then for all $i\le r$  the first cycle in the $i$-th permutation can be choosen between $(k+a_i)(k+a_i-1)\cdots(k+a_i-h+1)$ different cycles (only one choose if $h=0$). By multiplying over $i$ we get $f(k)\cdots f(k-h+1)$. We delete now all cycles labeled by $1$ and so we have elements in $S_{k-h+a_1}\times\cdots\times S_{k-h+a_r}$ that by induction are $a_{n-h}^{(t)}(f(1),\dots,f(k-h-1))$.

It is simple to check  by definition that $a_n^{(t)}(x_1,\dots,x_k)=a_n^{(t)}(x_1,\dots,x_{k-1})$ if $k\not\equiv n$ mod $t$ and that if $n\equiv k $ mod $t$ then \\$a_n^{(t)}(x_1,\dots,x_k)=\sum_{h\ge 0} x_kx_{k-1}\cdots x_{k-h+1}a_{n-h}^{(t)}(x_1,\dots,x_{k-h-1})$. 
Thus the proof is completed.
\end{proof}

If $t=1$, $a_n^{(1)}(x_1,\dots,x_k)=e_n(x_1,\dots,x_k)$. We then have
\begin{corollary}
\label{C:IntA}
For all $n,k\in \N$, $e_n(f(1),\dots,f(k))$ is the number of elements in $S_{k+1+a_1}\times\cdots\times S_{k+1+a_r}$ such that all permutations have the same ordered cycle structure up to $k+1-n$, each one with  $k+1-n+a_i$ cycles. 
\end{corollary}
In the particular case of the Jacobi-Stirling numbers of the first kind, by (\ref{E:JS}) we have the following new interpretation.

\begin{corollary}
\label{C:NewIntJS1}
For $n,k,z\in\N$ $Jc(n,k)$ is the number of elements in $S_n\times S_{n+z}$ whose permutations have the same ordered cycle structure up to $k$ and the number of cycles are respectively $k$ and $k+z$.
\end{corollary}
Obviously, we can apply Corollary \ref{C:NewIntJS1} to the Legendre-Stirling numbers of the first kind just by setting $z=1$.
For example, we know that $Lc(3,2)=8$. Indeed, applying Corollary \ref{C:NewIntJS1} we obtain the following elements of $S_3\times S_4$: $(12)(3),(12)(3)(4)$; $(12)(3),(13)(2)(4)$; $(12)(3),(14)(2)(3)$; $(13)(2),(12)(3)(4)$;\\ $(13)(2),(13)(2)(4)$; $(13)(2),(14)(2)(3)$; $(1)(23),(1)(23)(4)$; \ $(1)(23),(1)(24)(3)$.

It is possible to consider $a_n^{(t)}(f(1),\dots,f(k))$ as polynomial in  $a_1,\dots, a_r$. 

\begin{proposition}
\label{P:IntACoeff}
Let $r\in\PPP$, $\beta_1,\dots, \beta_n\in\N$ and $f(x)=(x+a_1)\cdots (x+a_r)$. Then for all $n,k\in\N$ and $t\ge1$ the coefficient of $a_1^{\beta_1}\cdots a_r^{\beta_r}$ in $a_n^{(t)}(f(1),\dots,f(k))$ is the number of elements in $S_{k+2}^r$ such that all permutations have the same ordered cycle structure up to $k+1-n$, each one with  $k+2-n$ cycles and if the $j$-th cycle is not trivial then $k+1-n\equiv j$ mod $t$; moreover, the $i$-th permutation has exactly $\beta_i+1$ big numbers.
\end{proposition}

The proof is the same as the one in Proposition \ref{P:IntA}. In the $i$-th permutation each big number give us a contribution of the indeterminate $a_i$.

In the case of the Jacobi-Stirling numbers, we consider the polynomial $f(x)=(x+a_1)(x+z)$ and consider only the monomial without $a_1$ (it is equivalent to set $a_1=0$). Then we have exactly one big number in the first permutation of $S_{k+2}^2$ as in Proposition \ref{P:IntACoeff}. This number is $k+2$ and it is necessarily in a trivial cycle. Therefore we can omit it and by applying (\ref{E:JS}) we have the following result.
\begin{corollary}
\label{C:NewIntJScoef}
For all $n,j,b\in\N$ the coefficient of $z^b$ in $Jc(n,j)$ is the number of elements in $S_{n}\times S_{n+1}$ such that both permutations have the same ordered cycle structure up to $j$, with respectively $j$ and $j+1$ cycles and the second permutation has exactly $k+1$ big numbers.
\end{corollary}

For example, if $n=3,j=2$, $Jc(3,2)=5+3z$. The elements of $S_3\times S_4$ whose second permutation has one big number are $(12)(3),(12)(3)(4)$; $(12)(3),(13)(2)(4)$;  $(13)(2),(12)(3)(4)$; $(13)(2),(13)(2)(4)$;  $(1)(23),(1)(23)(4)$; the other $3$ elements are $(1)(23),(1)(24)(3)$; $(12)(3),(14)(2)(3)$; $(13)(2),(14)(2)(3)$.

\subsection{Another interpretation of $h_n(f(1),\dots,f(k))$}
\label{SS:Anoth}
In the previous section we have a combinatorial interpretation of the Jacobi-Stirling numbers of the first kind, but not of the second kind. In this subsection we want to use an idea similar to the ordered cycle structure to obtain another combinatorial interpretation for the Jacobi-Stirling numbers of the second kind and, more  generally, of $h_n(f(1),\dots,f(k))$, with $f$ as usual.

Let we denote with $\overline S_n^h$ be the set of ordered sequences of $h$ elements, where the $i$-th element is an ordered finite sequence (maybe empty) of integers  in $\{i,i+1,\dots,n\}$. Given an element $s\in\overline S_n^h$ we say that $s$  has dimension $k$ if the sum of the cardinality of all the $h$ sequences is $k$.
For example, in $\overline S_4^3$ $(3,1,4),(4,4,3,4),()$ is an element of dimension $7$.

As done for the permutations, we say that two elements in $\overline S_m^h$ and $\overline S_n^h$ have the same \emph{ordered structure} if the cardinalities of the sequences of both elements are equal.
For example $(3,1,2),(4,4),(3)$ and $(1,2,1),(3,3),(3)$ have the same ordered structure.

\begin{proposition}
\label{P:AnInt} 
Let $r\in\PPP$, $a_1,\dots,a_r\in\N$ and $f(x)=(x+a_1)\cdots (x+a_r)$. Then for all $n,k\in\N$ $h_n(f(1),\dots, f(k))$ is the number of elements in $\overline S^k_{k+a_1}\times \cdots\times\overline S^k_{k+a_r}$ with the same ordered structure and dimension $n$.
\end{proposition}

\begin{proof}
If $n=0$ then we have only the empty sequence, indeed $h_0(f(1),\dots,f(k))=1$. If $k=0$ then $\overline S^0_n$ has the trivial element only if $n=0$. Indeed, $h_n(0)=0$ for $n\ge1$. Now let $n\ge 1$ and suppose that the thesis  true for smaller values of $n$. Fix an element $s=(s_i)_{i\le k}\in\overline S^k_{k+a_1}\times \cdots\times\overline S^k_{k+a_r}$. For all $i\le k$ the cardinality of the first sequence in $s_i$ is the same for all $i$ and it ranges between $0$ and $n$: we call $c$ this number. Now each number in the first sequence in $s_i$ is in $\{1,\dots, k+a_i\}$. We delete  the first sequence in each  $s_i$ and we decrease all numbers by one. We obtain elements in $\overline S^{k-1}_{k-1+a_1}\times \cdots\times\overline S^{k-1}_{k-1+a_r}$ with dimension $n-c$. By induction there are  $ h_{n-c}(f(1),\dots,f(k-1))$ of them. In formula, the cardinality of $\overline S_{k+a_1}\times \cdots\overline S_{k+a_r}$ is
\begin{align*}
\sum_{c=0}^{n}h_{n-c}(f(1),\dots,f(k-1))\prod_{i=1}^r (k+a_i)^{c} =&\sum_{c=0}^{n}h_{n-c}(f(1),\dots,f(k-1))f(k)^c\\=&h_n(f(1),\dots,f(k)). 
\end{align*}
Thus the proof is completed.
\end{proof}

By (\ref{E:JS}) we have
\begin{corollary}
\label{C:NewJS2}
Let $n,j,z\in\N$. Then $JS(n,j)$ is the number of elements in $\overline S^j_{j}\times \overline S^j_{j+z}$ with the same ordered structure and dimension $n-j$.
\end{corollary}

We now give an example when $z=1$. We know that $LS(3,2)=8$. The elements in $\overline S^2_{2}\times \overline S^2_{3}$ as in Corollary \ref{C:NewJS2} are the following: $(1)(),(1)()$; $(1)(),(2)()$; $(1)(),(3)()$; $(2)(),(1)()$; $(2)(),(2)()$; $(2)(),(3)()$; $()(2),()(2)$; $()(2),()(3)$.

If we consider $h_n(f(1),\dots,f(k)) $ as polynomial in  $a_1,\dots,a_r$ we have the following result.
\begin{proposition}
\label{P:AnIntCoeff} 
Let $r\in\PPP$, $\beta_1,\dots,\beta_r\in\N$ and $f(x)=(x+a_1)\cdots (x+a_r)$. Then for all $n,k\in\N$ the coefficient of $a_1^{\beta_1}\cdots a_r^{\beta_r}$ in $h_n(f(1),\dots, f(k))$ is the number of elements  $s=(s_i)_{i\le r}\in(\overline S_{k+1}^k)^r$ with the same ordered structure, dimension $n$ and such that $k+1$ appears $\beta_i$ times in all sequences of $s_i$ for all $i\le r$.
\end{proposition}
The proof is essentially the same as the one of Proposition \ref{P:AnInt}. In the case of the Jacobi-Stirling numbers of the second kind, this proposition becomes
\begin{corollary}
\label{C:AnIntJS}
Let $n,j,b\in\N$. Then the coefficient of $z^b$ in $JS_n^{(j)}(z)$ is the number of elements in $\overline S_{j}\times \overline S_{j+1}$ with the same ordered structure and dimension $n-j$, such that $j+1$ appears $b$ times.
\end{corollary}

\subsection{Monomial symmetric functions}
In this last subsection we want to give a simple combinatorial interpretation of the  monomial symmetric functions, evaluated in $f(1),\dots,f(k)$, where $f$ is as usual.

Let $\lambda=(\lambda_1,\dots,\lambda_t) $ be a partition of $n=\vert\lambda\vert$ and let $k$ be an integer, $k\ge t$. Then the monomial symmetric function associated to $\lambda$ in the indeterminates $x_1,\dots,x_k$ is
$$
m_\lambda(x_1,\dots,x_k)=\sum_\sigma x_{\sigma(1)}^{\lambda_1}\cdots x_{\sigma(k)}^{\lambda_k}
$$
where the sum is over the group of permutations $S_k$ modulo  the stabilizer of $\lambda$. We will give a combinatorial interpretation of $m_\lambda(f(1),\dots,f(k))$.  Let ${\bf u}=(u_1,\dots,u_k),{\bf v}=(v_1,\dots,v_k)\in \N^k$ be two sequences of $k$ integers. We say that ${\bf u}\ge {\bf v}$ if $u_i\ge v_i$ for all $i\le k$. Moreover, if $m\in\N$ we say that ${\bf u}\le m$ if $u_i\le m$ for all $i\le k$.

Fix  a positive integer $k$ and a partition $\lambda=(\lambda_1,\dots,\lambda_t) $, with $k\ge t$. Let $n=\vert \lambda\vert$. We define  $S_{\lambda,k}$ as the set of $n$-tuples of elements $\le k$ in increasing order such that $\{m(1),m(2),\dots,m(k)\}=\{\lambda_1,\dots,\lambda_k\}$ as multiset, where $m(i)$ denotes the multiplicity of $i$ and $\lambda_i=0$ for all $i>t$. For example, if $k=3,\lambda=(2,1)$ then $S_{\lambda,k}=\{(1,1,2),(1,1,3), (1,2,2),(1,3,3),(2,2,3),(2,3,3)\}$.
\begin{proposition}
\label{P:Intmonom}
Let $r\in\PPP$, $a_1,\dots,a_r\in\N$ and $f(x)=(x+a_1)\cdots(x+a_r)$. Then $m_{\lambda}(f(1),\dots,f(k))$ is the number of $r+1$ sequences $(s_0,\dots,s_r)$, not necessarily ordered, each one with $n=\vert \lambda\vert$ elements, such that $s_0\in S_{\lambda,k}$ and $s_0\le s_i\le k+a_i$ for all $i\le k$.
\end{proposition}

\begin{proof}
Let $s_0=(j_1,\dots,j_n)\in S_{\lambda,k}$. Then for all $i\le r$ there are $(k+a_i-j_1+1)(k+a_i-j_2+1)\cdots(k+a_i-j_n+1)$ sequences $s_i$ such that $s_0\le s_i\le k+a_i$. Therefore there are 
$$
\prod_{i=1}^r\bigg(\prod_{h=1}^n (k+a_i-j_h+1)\bigg)=\prod_{h=1}^n f(k-j_h+1)=\prod_{j=1}^{k} f(j)^{m(k+1-j,s_0)}
$$
sequences $(s_0,\dots,s_r)$ as in the statement, with $s_0$ fixed ($m(j,s_0)$ denotes the multiplicity of $j$ in $s_0$). Now we sum over $s_0$ and by virtue of definition of $S_{\lambda,k}$ we have that all such sequences are
$$
\sum_{s_0\in S_{\lambda,k}}\prod_{j=1}^{k} f(j)^{m(k+1-j,s_0)}=m_{\lambda}(f(1),\dots, f(k)).
$$
\end{proof}

When $\lambda=(1,1,\dots,1)$ we have that $m_\lambda(x_1,\dots,x_r)=e_{\vert\lambda\vert}(x_1,\dots,x_r)$. Therefore we have the following result.

\begin{corollary}
\label{P:IntmonomJs}
Let $n,j,z\in\N$. Then $Jc_n^{(j)}(z)$ is the number of  sequences $(s_0,s_1,s_2)$, not necessarily  ordered, each one with $n=\vert \lambda\vert$ elements, such that $s_0$ has its elements pairwise distinct and $s_0\le s_1\le k$ and $s_0\le s_2\le k+z$.
\end{corollary}
If we want to consider $m_\lambda(f(1),\dots,f(k))$ as polynomial in $a_1,\dots, a_r$ then we have the following interpretation.

\begin{proposition}
\label{P:IntmonomCoeff}
Let $r\in\PPP$, $\beta_1,\dots,\beta_r\in\N$ and $f(x)=(x+a_1)\cdots(x+a_r)$. Then the coefficient of $a_1^{\beta_1}\cdots a_r^{\beta_r}$ in $m_{\lambda}(f(1),\dots,f(k))$ is the number of $r+1$ sequences $(s_0,\dots,s_r)$, not necessarily  ordered, each one with $n=\vert \lambda\vert$ elements, such that $s_0\in S_{\lambda,k}$ and $s_0\le s_i\le k+1$ for all $i\le k$, such that $k+1$ appears $\beta_i$ times in $s_i$.
\end{proposition}

\begin{corollary}
\label{P:IntmonomJsCoeff}
Let $n,j,b\in\N$. Then the coefficient of $z^b$ in $Jc_n^{(j)}(z)$ is the number of  sequences $(s_0,s_1,s_2)$, not necessarily  ordered, each one with $n=\vert \lambda\vert$ elements, such that $s_0$ has its elements pairwise distinct and $s_0\le s_1\le k$ and $s_0\le s_2\le k+1$, such that $k+1$ appears $b$ times.
\end{corollary}

\section{Final remarks}
\label{S:Rem}
In this section we recall properties of symmetric functions and apply them to the special symmetric functions $H_{j,n}^f$ and $E_{j,n}^f$. It is well known (see e.g. \cite[Chapter I.2]{Macdonald1979}) that the generating functions of the elementary and complete symmetric functions are respectively
\begin{align*}
\sum_{r=0}^{n}e_r(x_1,\dots,x_n)t^r&=\prod_{i=1}^{n}(1+x_it) \\
\sum_{r=0}^{\infty} h_r(x_1,\dots,x_n) t^r&=\prod_{i=1}^{n}\frac{1}{1-x_i t}
\end{align*}
Therefore, the generating functions of $E_{j,n}^f$ and $H_{j,n}^f$ are respectively
\begin{align*}
\sum_{j=0}^n E_{j,n}^f t^{n-j}=\prod_{i=1}^{n-1}(1+f(i) t) \\
\sum_{n=j}^\infty H_{j,n}^f t^{n-j}=\prod_{i=1}^j\frac{1}{1-f(i)t}
\end{align*}
or, equivalently,
\begin{align}
\sum_{j=0}^n E_{j,n}^f t^{j}=t\prod_{i=1}^{n-1}(t+f(i)) \label{E:GFEbis}\\
\sum_{n=j}^\infty H_{j,n}^f t^{n}=\prod_{i=1}^j\frac{t}{1-f(i)t}.\label{E:GFHbis}
\end{align}

We have that the matrices $H^f=(H_{j,n}^f)_{j,n\in\N}$ and $E^f=((-1)^{j+n}E_{j,n}^f)_{j,n\in \N}$ are  inverses of each other. In fact,  for fixed $j,n'\in \N$
\begin{align}
\frac{\prod_{i=1}^{n'-1}(1-f(i)t)}{\prod_{l=1}^{j}(1-f(l)t)}&=\bigg(\sum_{n=0}^\infty H_{j,n}^f t^{n-j} \bigg)\bigg(\sum_{j'=0}^{n'} E_{j',n'}^f (-t)^{n'-j'} \bigg) \label{E:prodHE1}\\
&=\sum_{n=0}^\infty\sum_{j'=0}^{n'} (-1)^{j'+n'}H_{j,n}^f E_{j',n'}^f t^{n-j+n'-j'}. \label{E:prodHE2}
\end{align}
When we extract the coefficient of $t^{n'-j}$, by (\ref{E:prodHE2}), we get the entry $j,n'$ of the product $H^fE^f$ . If $n'<j $ then  $t^{n'-j}$ has coefficient $0$ in the RHS of (\ref{E:prodHE1}); if $n'= j$  the LHS of (\ref{E:prodHE1}) is $\frac{1}{1-f(j)t} $ and therefore the coefficient of $t^0$ is $1$; if $n'>j$ the LHS of (\ref{E:prodHE1}) is a polynomial of degree $n'-j-1$ and therefore the coefficient of $t^{n'-j}$ is $0$. So the product $H^fE^f$ is the (infinite) identity matrix.

Now, for all $j\ge 1$, set $\langle x \rangle_j:=x(x-f(1))(x-f(2))\cdots (x-f(j-1))$ and set $\langle x \rangle_0:=1$. By previous remark and by (\ref{E:GFEbis}) we have
\begin{align}
x^n&=\sum_{j=0}^n H_{j,n}^f \langle x \rangle_j\label{E:bH}\\
\langle x \rangle_n&=\sum_{j=0}^n E_{j,n}^f x^j.\label{E:bE}
\end{align}

Finally, since $f$ has all real and nonpositive zeros, $f$ is injective if it is evaluated on $\R^+$. Thus, it is possible to use the Newton interpolation formula 
\begin{equation*}
\label{E:Nif}
x^n=\sum_{j=0}^n\Bigg(\sum_{r=0}^{j}\frac{x_r^n}{\displaystyle\prod_{k=0, k\ne r}^j(x_r-x_k)}\Bigg)\displaystyle\prod_{i=0}^{j-1}(x-x_i)
\end{equation*} 
in (\ref{E:bH}) to obtain the following expression for the $H_{j,n}^f$ when $n\ge 1$ (we set $x_0=0$):
\begin{equation*}
\label{E:expl}
H_{j,n}^f=\sum_{r=1}^{j}\frac{f(r)^{n-1}}{\displaystyle\prod_{k=1, k\ne r}^j(f(r)-f(k))}.
\end{equation*}

\end{document}